\documentclass{amsart}

\newtheorem{thm}{Theorem}[section]

\newtheorem{lem}[thm]{Lemma}
\newtheorem{prop}[thm]{Proposition}

\theoremstyle{definition}
\newtheorem{defi}[thm]{Definition}

\newtheorem{rmq}[thm]{Remark}

\usepackage{amsmath,amsfonts,amssymb,amsthm,latexsym,bbm}

\usepackage{mathrsfs}

\newcommand{\1}{\mathbbm 1}

\newcommand{\R}{\mathbb R}
\newcommand{\Z}{\mathbb Z}
\renewcommand{\P}{\mathbb P}
\newcommand{\E}{\mathbb E}

\renewcommand{\L}{\mathcal L}
\newcommand{\lap}{\mathcal L}
\newcommand{\A}{\mathcal A}
\newcommand{\alap}{\mathcal A}

\renewcommand{\root}{\mathcal R}

\newcommand{\ie}{{\it i.e. }}

\newcommand{\point}{\!\cdot\!}
\DeclareMathAlphabet{\mathpzc}{OT1}{pzc}{m}{it}

\newcommand{\la}{\langle}
\newcommand{\ra}{\rangle}

\begin{document}

\title{Affine Dunkl Processes}

\author{Fran\c cois Chapon} 
\address{Laboratoire de probabilit\'es et Mod\`eles Al\'eatoires, Universit\'e Paris 6 Pierre et Marie Curie, 4 place Jussieu, 75252 Paris Cedex 05, France}
\email{francois.chapon@upmc.fr}

\renewcommand{\subjclassname}{%
  \textup{2010} Mathematics Subject Classification}

\begin{abstract}
We introduce  the analogue of Dunkl processes in the case of an affine root system of type $\widetilde{\text{A}}_1$. The construction of the affine Dunkl process is achieved by a skew-product decomposition by means of its radial part and a jump process on the affine Weyl group, where the radial part of the affine Dunkl process is defined as the unique solution of some stochastic differential equation. We prove that the affine Dunkl process is a c\`adl\`ag Markov process as well as a local martingale, study its jumps, and give a martingale decomposition, which are properties similar to those of the classical Dunkl process. 
\end{abstract} 
\maketitle

\section{Introduction}

The aim of the following is to study the analogue of Dunkl processes in the case of an affine root system of type $\widetilde{\text{A}}_1$ (root systems and Dunkl processes will be properly defined in the sequel). The affine Dunkl process $Y=(Y_t)_{t\geq0}$ with parameter $k$ will be defined as the Markov process in $\R$ with infinitesimal generator given by
\[
\A u(x)=\frac{1}{2}u''(x)+k\pi\cot(\pi x)u'(x)+\frac{k}{2}\sum_{p\in\Z}\frac{u(s_p(x))-u(x)}{(x-p)^2},
\]
acting on $u\in C^2_b(\R)$ and $x\in\R\setminus\Z$, where $k$ is a real number satisfying $k\geq\frac12$ and $s_p$ is the reflection around $p\in\Z$, \ie $s_p(x)=-x+2p$. It is a Markov process with jumps, whose radial part is the continuous Feller process living in the interval $]0,1[$, and with infinitesimal generator 
\[
\L u(x)=\frac{1}{2}u''(x)+k\pi\cot(\pi x)u'(x),
\]
for $u\in C([0,1])\cap C^2(]0,1[)$. The main idea to achieve the construction of the process $Y$ with generator $\A$ is to consider a skew-product decomposition, by starting from its radial part and adding the jumps successively at random times, the jump part of the process being  given by some process on the affine Weyl group associated to the root system $\widetilde{\text{A}}_1$.

In this introduction, we now
recall some facts about root system and Weyl group in the classical case as in the affine case, and some results about Dunkl processes.

\subsection{Root systems and reflection groups}

All these facts  can be found in the book \cite{humphr}. Let $V$ be  a real euclidean  space of finite dimension  endowed with an inner product $\langle \cdot , \cdot \rangle$. For $\alpha\in V$, we denote by $\sigma_\alpha$ the orthogonal reflection associated  to the vector $\alpha$, which writes
\[
\sigma_\alpha(x)=x-2\frac{\la\alpha,x\ra}{\la\alpha,\alpha\ra}\alpha,
\]
for $x\in V$, and by $\mathcal H_\alpha=\{x\in V\, | \, \langle x,\alpha\rangle=0\}$ the hyperplane orthogonal to $\alpha$.
Let $\root^0\subset V$ be a crystallographic root system, which by definition is a finite set which satisfies
\begin{enumerate}
\item $\root^0$ spans $V$,
\item For all $\alpha\in\root^0$, $\sigma_\alpha(\root^0)=\root^0$,
\item For all $\alpha,\beta\in\root^0$, $\langle\alpha^\vee,\beta\rangle\in\Z$,
\end{enumerate}
where elements of $\root^0$ are called \emph{roots} and  the \emph{coroot} $\alpha^\vee$ is defined by $\alpha^\vee=2\frac{\alpha}{\la\alpha,\alpha\ra}$, so  $\langle \alpha^\vee,\alpha\rangle=2$ for all $\alpha\in\root^0$. The \emph{rank} of $\root^0$ is defined as  the dimension of $V$.
The \emph{Weyl group} $W^0$ associated to $\root^0$ is the subgroup of the orthogonal group of $V$, generated by reflections $\{\sigma_\alpha\, |\, \alpha\in\root^0\}$. Note that $W^0$ is a finite group for all root system in $\R^n$. Each root system can be written as a disjoint union $\root^0=\root^0_+\cup(-\root^0_+)$, where $\root^0_+$ and $-\root^0_+$ are separated by a hyperplane $\{x\in V\, |\, \langle \beta , x \rangle =0\}$, where $\beta$ is an arbitrary chosen vector in $V$ with $\beta\not \in \root^0$. $\root^0_+$ is called a positive subsystem. Let
\[
C=\{x\in V |\ \forall\, \alpha\in \root^0_+, \langle \alpha, x\rangle >0\}
\]
be the \emph{Weyl chamber}. We denote by $\overline C$ its closure, and by $\partial C$ its boundary which is a union of hyperplanes $\mathcal H_\alpha$, which are called the \emph{walls} of $C$. We have that $\overline C$ is a fundamental domain for the action of $W^0$ on $V$, \ie 
$W^0$ permutes the chambers of the system, where chamber means any connected component of $V\setminus \bigcup_{\alpha\in\root^0}\mathcal H_\alpha$.

All root systems of $\R^n$ have been classified, see \cite{humphr} for details. Let us just mention for example the $\text{A}_{n-1}$ root system (for $n\geq2$) which is the set 
\[
\{\pm (e_i-e_j)\, |\, 1\leq i<j\leq n\},
\]
where $(e_i)_{1\leq i\leq n}$ is the standard basis of $\R^n$, and $V$ is the orthogonal complement of the vector $e_1+\cdots+e_n$. Thus the rank of $\text{A}_{n-1}$ is $n-1$. The Weyl group in this case is isomorphic to the symmetric group $\mathcal S_n$, and the Weyl chamber is the cone
\[
\{(x_1,\ldots,x_{n-1})\in\R^{n-1}\,  |\, 0<x_1<\cdots < x_{n-1}\}.
\]

We now want to consider not only orthogonal reflections (leaving the origin of $V$ fixed), but also \emph{affine reflections} relative to hyperplanes which do not necessarily pass through the origin. To this end, given a root system $\root^0$, we 
 define the  corresponding \emph{affine root system} as the direct product $\root=\root^0\times\Z$. For  $\lambda=(\alpha,p)\in\root$, with $\alpha\in\root^0$ and $p\in\Z$,  we define for all $x\in V$
\[
\langle\lambda,x\rangle:=\la\alpha,x\ra-p,
\]
and the affine reflection associated to $\lambda$ by
\[
s_\lambda(x)=s_{(\alpha,p)}(x)=x-\la\alpha,x\ra\alpha^\vee+p\alpha^\vee.
\]
The positive affine root system is also defined by
\begin{equation}\label{explicitaffine}
\root_+=\{(\alpha,0)\, |\, \alpha\in\root^0_+\}\cup\{(\alpha,p)\, |\, \alpha\in\root^0, p\leq -1\}.
\end{equation}
The \emph{affine Weyl group} $W$ is the subgroup of the affine group of $V$, generated by affine reflections $\{s_\lambda\, |\, \lambda\in\root\}$. One can show that $W$ is the semidirect product of $W^0$ and the translation group corresponding to the lattice generated by the coroots of $\root^0$. For each $\lambda=(\alpha,p)\in\root$, we define the affine hyperplane associated to $\lambda$ by
\[
\mathcal H_\lambda=\{x\in V\, |\, \la \alpha, x\ra =p\}.
\]
Let $\mathscr{A}$ be the collection of all connected components of $V^\circ:=V\setminus\bigcup_{\lambda\in\root}\mathcal H_\lambda$.  Each element of $\mathscr A$ is called an \emph{alcove}. As for chambers, we single out one particular alcove
\[
\mathcal A_0=\{x\in V\, |\, 0<\la \alpha, x\ra <1 \text{ for all } \alpha\in\root^0_+\},
\]
called the \emph{principal alcove}. Then, we have that the affine Weyl group $W$ permutes the collection $\mathscr A$ of all alcoves transitively, and the principal alcove $\mathcal A_0$ is a fundamental domain for the action of $W$ on $V$.

In the case of the  rank one affine root system, denoted $\widetilde{\text{A}}_1$, this reduces to $\root^0=\{\pm\alpha\}$, $\root^0_+=\{\alpha\}$, and the positive affine root system is given by
\begin{equation}\label{positiveaffinerootdef}
 \root_+=\root^0_+\cup \{(\pm 1,p)\  | \ p\leq -1\}.
 \end{equation}
where we have identified  $\alpha$ with $1$ and $\alpha^\vee$ with $2$, so that $\langle \alpha^\vee,\alpha\rangle=2$. Hence, we will use the notation for affine reflections
\[
s_p(x)=-x+2p,
\]
for $p\in\Z$. The affine Weyl group is then in the $\widetilde{\text{A}}_1$  case isomorphic to the infinite dihedral group, and the principal alcove is just the interval $]0,1[$, so we will use in what follows either alcove or interval to describe an interval of the form $]n,n+1[$ with $n\in\Z$.

\subsection{Dunkl processes}

A good survey of Dunkl operators and processes can be found in the book \cite{livredunkl}. 

From an analytic point of view, the theory was initiated by Dunkl (\cite{dunkldifference}) who studied differential-difference operators defined by
\[
T_\xi u(x)=\partial_\xi u(x)+\sum_{\alpha\in\root_+^0}k(\alpha)\langle \alpha,\xi\rangle \frac{u(x)-u(\sigma_\alpha (x))}{\la x,\alpha\ra},
\]
for $\xi\in V$, where $\partial_\xi$ denotes the directional derivative corresponding to $\xi$, and where $u\in C^1(V)$ and $k$ is a nonnegative multiplicity function invariant by the Weyl group $W^0$ associated with $\root^0$, \ie $k\colon \root^0\to\R_+$ and $k\circ w=k$ for all $w\in W^0$. One of the most important property of these Dunkl operators is the fact that they commute, and this is at the basis of a rich analytic structure related to them.

The Dunkl Laplacian $\L^0$ is defined by 
\[
\L^0=\sum_{i=1}^n T_i^2, 
\]
and has explicit expression given by
\[
\mathcal L^0 u(x)=\frac{1}{2}\Delta u(x)+\sum_{\alpha\in\root^0_+} k(\alpha)\left(\frac{\la \nabla u(x),\alpha\ra}{\la x,\alpha\ra}+\frac{|\alpha|^2}{2}\frac{u(\sigma_\alpha (x))-u(x)}{\la x, \alpha\ra^2 }\right),
\]
acting on $u\in C^2_b(V)$, for $x\in V\setminus\bigcup_{\alpha\in\root^0} \mathcal H_\alpha$, where $C^2_b$ means continuous twice differentiable bounded functions. This generalizes the radial part of the Laplace-Beltrami operator of a Riemannian symmetric space of Euclidian type obtained when $k$ takes only certain values. 

From a probabilistic point of view, the study of Dunkl processes was originated by R\"osler and Voit  in \cite{roslervoit}, and then extensively   studied by Gallardo and Yor in \cite{gallardoyor1, gallardoyor2, gallardoyor3} and Chybiryakov \cite{chyb}. The Dunkl processes are a family of c\`adl\`ag Markov processes with infinitesimal  generator
$\L^0$, and   parameter $k$. Note that by its explicit expression given above, $\L^0$ does not depend of the choice of $\root^0$. Fixing a Weyl chamber $C$, the radial part of the Dunkl process is the projection of the Dunkl process by the canonical projection of $V$ onto the space $V/W^0$ of $W^0$-orbits in $V$, and we  can  identify  $V/W^0$ with $\overline C$ since $\overline C$  is a fundamental domain for the action of $W^0$.  The radial Dunkl process is then a diffusion process  with infinitesimal generator given by
\[
\L^{0,W^0} u(x)=\frac{1}{2}\Delta u(x)+\sum_{\alpha\in\root^0_+} k(\alpha)\frac{\la \nabla u(x),\alpha\ra}{\la x,\alpha\ra},
\]
for $u\in C_0^2(\overline C)$, the set of $C^2$ functions in $C$, continuous on $\overline C$, which vanish on the boundary of $\overline C$, and such that $\la\nabla u(x),\alpha\ra=0$ for $x\in\mathcal H_\alpha$, $\alpha\in\root^0_+$. Note that $\L^{0,W^0}$ is obtained from $\L^0$ acting on functions that are invariant by $W^0$. As an example,   when $k(\alpha)\equiv1$, the radial Dunkl process is the Brownian motion process in a Weyl chamber as studied by Biane, Bougerol and O'Connell \cite{bbo}. One can show that the radial Dunkl process is the unique solution of the following stochastic differential equation
\[
dX^0_t=dB_t+\sum_{\alpha\in\root^0_+}k(\alpha)\alpha\frac{dt}{\la X^0_t,\alpha\ra },
\]
where $(B_t)_{t\geq0}$ is a $n$-dimensional Brownian motion, and $X^0_0\in C$ almost surely. Furthermore, it has been  shown (\cite{livredunkl}) that, when $k(\alpha)\geq\frac{1}{2}$ for all $\alpha\in\root^0$, $X^0$  lives in $C$ almost surely, \ie $X^0$ never touches the walls of the chamber $C$.
 Now, we list some of the main properties of the Dunkl process, which is denoted $(Y^0_t)_{t\geq0}$, all of them can be found in \cite{livredunkl}. First, we see from its explicit expression,  that $\L^0$
decomposes into a continuous part and a jump part driven by the term $\sum_{\alpha\in\root^0_+} k(\alpha)\frac{u(\sigma_\alpha (x))-u(x)}{\la x, \alpha\ra^2 }$. One can show that the number of jumps is almost surely finite in any finite time interval, and when a jump occurs at some time $t$, there is a random reflection $\sigma_\alpha$ such that $Y^0_t=\sigma_\alpha Y^0_{t^-}$, hence $Y^0$ jumps from chamber to chamber. A remarkable property is that the Dunkl process is the first known example of Markov process with jumps which enjoy the time-inversion property like Brownian motion or Bessel processes. Another property, which will be of importance for the next, is the skew-product decomposition of the Dunkl process found in \cite{chyb}. This is a constructive way to define $Y^0$ starting from its radial part, by adding successively jumps in the direction of the roots, see \cite{chyb}  or \cite{livredunkl} for details.

Finally, we also mention that the counterpart of Dunkl processes in the negatively curved setting, which are called Heckman-Opdam processes, is investigated by Schapira in \cite{schapiraheckman}. They are Markov process with jumps, with infinitesimal generator given by
\begin{align*}
\mathcal D f(x)=\frac{1}{2}\Delta f(x) +\sum_{\alpha\in\root^0_+}k(\alpha)&\coth\frac{\la\alpha,x \ra}{2}\partial_\alpha f(x) \\ &+\sum_{\alpha\in\root^0_+} \frac{|\alpha|^2}{4\sinh^2\frac{\la\alpha,x\ra}{2}}\left(f(\sigma_\alpha(x))-f(x)\right),
\end{align*}
for $f\in C^2_b(V)$ and $x\in V\setminus\bigcup_{\alpha\in\root^0}\mathcal H_\alpha$.

\subsection{Heuristics of the affine Dunkl process}

Since the link between the operator $\A$ defined in the introduction and affine root system is not so obvious, we give in this section some little heuristics, without being rigorous.
We have seen in the previous section that the Dunkl process is the Markov process with infinitesimal generator 
\[
\mathcal L^0 u(x)=\frac{1}{2}\Delta u(x)+\sum_{\alpha\in\root^0_+} k(\alpha)\left(\frac{\la \nabla u(x),\alpha\ra}{\la x,\alpha\ra}+\frac{|\alpha|^2}{2}\frac{u(\sigma_\alpha (x))-u(x)}{\la x, \alpha\ra^2 }\right),
\]
acting on $u\in C^2_b(V)$ with $x\in V\setminus\bigcup_{\alpha\in\root^0} \mathcal H_\alpha$, and 
in the case of the rank one root system $\text{A}_1$ the Dunkl Laplacian reduces to 
\[
\mathcal L^0 u(x)=\frac{1}{2}u''(x)+k\frac{u'(x)}{x}+k\frac{u(-x)-u(x)}{2x^2},
\]
where $k$ is a nonnegative real number.
The main idea is then to replace the positive root system $\root^0_+$ by the positive affine root system $\root_+$ given by (\ref{positiveaffinerootdef}),
 and hence to define the affine Dunkl Laplacian in the $\widetilde{\text{A}}_1$ case by
\begin{multline*}
\alap u(x)=\frac12 u''(x)+k\frac{u'(x)}{x}+ k\frac{u(-x)-u(x)}{2x^2} +k\sum_{p\leq -1}\left\{\frac{u'(x)}{x-p}+\frac{u'(x)}{x+p}\right.\\\left.+\frac{u(-x+2p)-u(x)}{2(x-p)^2}+\frac{u(-x-2p)-u(x)}{2(x+p)^2} \right\}.
\end{multline*}
Recall the series expansion 
of the cotangent function 
\[
\pi\cot(\pi x)=\frac{1}{x}+\sum_{n\geq 1}\left(\frac{1}{x+n}+\frac{1}{x-n}\right),
\]
for $x\in\R\setminus\Z$, see \cite{astuce}, which can be written more elegantly 
\[
\pi\cot (\pi x)=\sum_{n\in\Z}\frac{1}{x-n}.
\]
Note that the latter formula is quite dangerous since it is not absolutely convergent, and we have to be cautious with  the summation order. Hence, using the cotangent expansion, the affine Dunkl Laplacian writes
\[
\A u(x) =\frac{1}{2}u''(x)+k\pi\cot(\pi x) u'(x)+\frac{k}{2}\sum_{p\in \Z}\frac{u(s_p(x))-u(x)}{(x-p)^2}.
\]

The following is divided in two parts. In the first one, corresponding to section~\ref{sectionradialaffine}, we define the radial affine Dunkl process as the unique solution of some stochastic differential equation. We calculate its semigroup and study some of its properties. The second part, which is section~\ref{sectionaffinedunkl}, is devoted to the construction of the affine Dunkl process with generator $\A$, using a skew-product decomposition by means of the radial process and a pure jump process on the affine Weyl group. We study its jumps and also give a martingale decomposition.

\section{The radial affine Dunkl process}\label{sectionradialaffine}

\subsection{Definition of the radial process}

In what follows, the parameter $k$ of the affine Dunkl process is a real number satisfying $k\geq\frac{1}{2}$.
We start by studying the radial part of the  affine Dunkl Process, which is the process $(X_t)_{t\geq0}$ solution to the following stochastic differential equation.

\begin{prop}\label{eqradial}
The stochastic differential equation
\[
dX_t=dB_t+k\pi\cot(\pi X_t)dt,
\]
with initial condition $X_0=x\in]0,1[$ a.s., and
where $(B_t)_{t\geq0}$ is a standard brownian motion, admits a unique strong solution.
Furthermore, let $S$ be the first exit time of the interval $]0,1[$, that is
\[
S=\inf\{t\geq0\, |\, X_t=0\text{ or }1\}.
\]
Then  $\P(S=+\infty)=1$ if and only if $k\geq \frac{1}{2}$. 
\end{prop}
\begin{proof}
The operator 
\[
\lap=\frac{1}{2}\frac{d^2}{dx^2}+k\pi\cot(\pi x)\frac{d}{dx}
\]
acting on $C([0,1])\cap C^2(]0,1[)$
generates a Feller semigroup on $C([0,1])$, hence the corresponding stochastic differential equation admits a unique strong solution \cite{ek}.  

For the second assertion, 
we use the standard scale function technique as in \cite{kara}. The scale function $p$ is defined by
\[
p(x)=\int_c^x \exp\Big(-2k\int_c^\xi \pi\cot(\pi y)dy\Big)d\xi,
\]
for $x\in]0,1[$. We can choose $c=1/2$ without loss of generality. Then, 
\[
p(x)=\frac1\pi \int_{\pi/2}^{\pi x} \frac{1}{(\sin y)^{2k}}dy.
\]
Since $\frac{1}{(\sin y)^{2k}}\sim\frac{1}{x^{2k}}$ for $x=0$, $p$ diverges for $k\geq\frac12$.
Indeed, we have $p(0)=-\infty$ and $p(1)=+\infty$. This guarantees that for $k\geq\frac12$, $S=+\infty$ a.s., according to \cite{kara}.

Now, suppose that $k<\frac12$. Define the speed measure by
\[
m(dx)=\frac{2}{p'(x)}dx=2(\sin (\pi x))^{2k}dx.
\]
Define also
\[
v(x)=\int_c^x(p(x)-p(y))m(dy).
\]
Then by Feller's test for explosion, see theorem $5.29$ in \cite{kara}, $S=+\infty$ a.s. if and only if $v(0)$ and $v(1)$ equal $\pm\infty$.
Since $p(x)<+\infty$ for all $x\in[0,1]$, we have to look at the finiteness of $\int_c^xp(y)m(dy)$. But,
\begin{align*}
\Big|\int_{1/2}^{x}p(y)m(dy)\Big|&=\Big|\frac{2}{\pi^2}\int_{\pi/2}^{\pi x}dy\int_{\pi/2}^{\pi y}d\xi \frac{1}{(\sin \xi)^{2k}}(\sin y)^{2k}\Big|\\
& \leq \Big|\frac{2}{\pi^2}\int_{\pi/2}^{\pi x}dy\int_{\pi/2}^{\pi y}d\xi \frac{1}{(\sin \xi)^{2k}}\Big|\\
&=\Big|\frac{2}{\pi^2}\int_{\pi/2}^{\pi x}\frac{x-\pi/2}{(\sin\xi)^{2k}}d\xi\Big|.
\end{align*}
Hence, for $k<\frac12$, $\int_{c}^xp(y)dy$ is finite for $x=0$ and $x=1$, and $S<+\infty$ a.s.
\end{proof}

\begin{rmq} \label{eqradialremark}
If we let the starting point live in some interval $]n,n+1[$, with $n\in\Z$, by the periodicity of the cotangent function we get exactly the  same kind of result, that is, the process $X$ lives in $]n,n+1[$ almost surely.
\end{rmq}

Hence, we define
\begin{defi}
The continuous Feller process $(X_t)_{t\geq0}$ with infinitesimal generator
\[
\L u(x)=\frac{1}{2}u''(x)+k\pi\cot(\pi x)u'(x)
\]
acting on $u\in C(\overline I)\cap C^2(I)$, where $I$ is some interval $I=]n,n+1[$, $n\in\Z$, and $X_0\in I$ a.s., is called the 
\emph{radial affine Dunkl process} with parameter $k$.
\end{defi}

Let us mention that for 
$k=1$, the radial affine Dunkl process is the  Brownian motion conditioned to stay in the interval $]0,1[$, also known as the Legendre process (see \cite{revuzyor}). Indeed, in that case, the generator of the process writes
\[
\L u(x)=\frac{1}{2}u''(x) +\frac{h'(x)}{h(x)}u'(x),
\]
where $h(x)=\sin(\pi x)$ is an eigenfunction for the Laplacian $\Delta$, and hence the corresponding process $X$ is a Doob $h$-transform at the bottom of the spectrum of Brownian motion killed when it reached the walls of $]0,1[$. Let us also mention that Brownian motions in alcoves are related to the process of eigenvalues of the Brownian motion with values in the special unitary group $SU(n)$, see \cite{bianepolya}. The last two remarks are analogues of the same kind of properties for the radial Dunkl process in the classical case, see \cite{livredunkl}.

The property of recurrence of the radial affine Dunkl process is the content of the next proposition.
\begin{prop} Let $k\geq\frac12$. 
The radial affine Dunkl process $(X_t)_{t\geq0}$ is recurrent, that is for every $y\in]0,1[$, we have
\[
\P(\exists\ 0\leq t<+\infty, \ X_t=y)=1.
\]
\end{prop}
\begin{proof}
Let $S$ be the first exit time of the interval $]0,1[$, and 
\[
p(x)=\frac1\pi \int_{\pi/2}^{\pi x} \frac{1}{(\sin y)^{2k}}dy,
\]
be the scale function as in the proof of the last proposition. Since $p(0)=-\infty$ and $p(1)=+\infty$ for $k\geq\frac12$,   we have
\[
\P(S=+\infty)=\P\Big(\sup_{t\geq0} X_t=1\Big)=\P\Big(\inf_{t\geq0}X_t=0\Big)=1.
\]
Hence by continuity of the paths of $(X_t)_{t\geq0}$, the result follows.
\end{proof}

From now on, we will denote by $\P_x$ the distribution of the radial affine Dunkl process starting from $x\in\R\setminus\Z$.

\subsection{Semigroup of the radial affine Dunkl process}

Orthogonal polynomials theory (see \cite{schou}) will allow us to make explicit the semigroup of the radial affine Dunkl process. 

First, let us recall some standard facts about Gegenbauer polynomials, which can be found in \cite{magnus} or \cite{szego} for example. Gegenbauer polynomials $G_n^{(k)}$ (which can be expressed in terms of  Jacobi polynomials $P_n^{(k-\frac{1}{2},k-\frac{1}{2})}$) are orthogonal for the weight $(1-x^2)^{k-\frac{1}{2}}\1_{[-1,1]}(x)dx$, \ie
\[
\int_{[-1,1]}G_n^{(k)}(x)G_m^{(k)}(x)(1-x^2)^{k-\frac{1}{2}} dx=\pi(\omega_n^{(k)})^{-1}\delta_{n,m},
\]
for $k>-\frac{1}{2}$, where
\[
\omega_n^{(k)}=\frac{n!(k+n)\Gamma(k)^2}{2^{1-2k}\Gamma(n+2k)}.
\]
The first polynomials are (for $k\not=0$), $G^{(k)}_0(y)=1$, $G^{(k)}_n(y)=2ky$,\dots
They are of the same parity than $n$, and $G^{(k)}_n(-y)=(-1)^nG^{(k)}_n(y)$. The value at 1 is $G^{(k)}_n(1)=\frac{\Gamma(2k+n)}{n!\Gamma(2k)}$. For $k>0$, they admit an explicit expression,  given by
\begin{equation} \label{explicitgegen}
G^{(k)}_n(y)=\frac{1}{\Gamma(k)}\sum_{m=0}^{\lfloor \frac{n}{2}\rfloor} (-1)^m\frac{\Gamma(k+n-m)}{m!(n-2m)!}(2y)^{n-2m}.
\end{equation}
Furthermore, for all $n\geq0$, $G_n^{(k)}$ is the unique polynomial solution of the equation
\[
(1-x^2)f''(x)-(2k+1)xf'(x)+n(n+2k)f(x)=0.
\]

Now, we can state the following proposition, which gives the semigroup of the radial affine Dunkl process.
\begin{prop}
The radial affine Dunkl process $(X_t)_{t\geq0}$, with $X_0\in]0,1[$ a.s., is the Markov process with semigroup on $]0,1[\times]0,1[$ given by
\[
q_t(x,y)dy=\sum_{n\geq0}e^{-\lambda_n t}G_n^{(k)}(\cos\pi x)G_n^{(k)}(\cos\pi y)\omega_n^{(k)}(\sin\pi y)^{2k}dy,
\]
where the $G_n^{(k)}$ are the Gegenbauer polynomials, with eigenvalues  $\lambda_n=\frac{\pi^2}{2}n(n+2k)$, and $\omega_n^{(k)}$ is a  normalization constant given by
\[
\omega_n^{(k)}=\frac{n!(k+n)\Gamma(k)^2}{2^{1-2k}\Gamma(n+2k)}.
\]
\end{prop}
\begin{proof}
Since for all $n\geq0$, $G_n^{(k)}$ is the unique polynomial solution of the equation
\[
(1-x^2)f''(x)-(2k+1)xf'(x)+n(n+2k)f(x)=0,
\]
we have that $G_n^{(k)}(\cos(\pi x))$ is solution of the equation
\[
\frac{1}{2}g''(x)+k\pi\cot (\pi x) g'(x)=-\frac{\pi^2}{2}n(n+2k) g(x),
\]
for $x\in[0,1]$. The weight associated is then the measure $\pi(\sin \pi x)^{2k}\1_{[0,1]}dx$, and by \cite{schou}, we obtain that the semigroup associated to the diffusion of generator $\L=\frac{1}{2}\frac{d^2}{dx^2}+k\pi\cot(\pi x)\frac{d}{dx}$, is given by
\[
q_t(x,y)dy=\sum_{n\geq0}e^{-\lambda_n t}G_n^{(k)}(\cos\pi x)G_n^{(k)}(\cos\pi y)\omega_n^{(k)}(\sin\pi y)^{2k}dy. \qedhere
\]
\end{proof}

\subsection{Some properties of the radial process}

As we will see, an important functional of the radial affine Dunkl process, is the continuous process $\frac{1}{\sin^2(\pi X_\cdot)}$. First, note that since $X$ is continuous and never reaches (for $k\geq\frac{1}{2}$) the walls of the alcove where is started from, we have that for all $t\geq0$, there exists some random $\varepsilon_t>0$ such that
\[
\inf_{s\in[0,t]} \sin^2(\pi X_s)>\varepsilon_t, \quad \text{ $\P_x$-a.s.}
\]
Hence, we obtain that for all $t\geq 0$,
\[
\int_0^t \frac{ds}{\sin^2(\pi X_s)}< +\infty, \quad \text{ $\P_x$-a.s.}
\]
Note also that, since $\sin^2(\pi X_s)<1$ for all $s\geq0$, we have that $\int_0^t\frac{ds}{\sin^2(\pi X_s)}\geq t$, so
\[
\int_0^t \frac{ds}{\sin^2(\pi X_s)} \underset{t\to+\infty}{\longrightarrow}+\infty, \quad \text{ $\P_x$-a.s.} 
\]

To study some properties of  the radial affine Dunkl process and more particularly of the last functional, we will need a few lemmas. Let us introduce some standard notations. The binomial coefficient is denoted $\binom{n}{k}=\frac{n!}{k!(n-k)!}$, for all integers $n$ and $k$, the rising factorial (also known as Pochhammer's symbol) is defined by
\[
(\alpha)_n=\alpha(\alpha+1)\cdots(\alpha+n-1)=\frac{\Gamma(\alpha+n)}{\Gamma(\alpha)},
\]
for real $\alpha$, where $\Gamma$ is the usual Gamma function, and the falling factorial is defined by
\[
[\alpha]_n=\alpha(\alpha-1)\cdots(\alpha-n+1),
\]
with $(\alpha)_n=[\alpha]_n=0$ by convention. 
The relation between rising and falling factorials  is given by $(\alpha)_n=[\alpha+n-1]_n$. We have the following  classical lemma.
\begin{lem}\label{lemmaiwa}
For all $\alpha\in\R$, and all $n\geq0$, we have
\[
\sum_{l=0}^n (-1)^l \binom{n}{l}[\alpha+l]_i=0,
\]
for  $i=0,\ldots,n-1$.
\end{lem}
\begin{proof}
This can be done 
by looking at the function $u(x)=x^\alpha(1-x)^n$ for which 1 is a zero of order $n$, and by differentiating $u$ $i$-times.
\end{proof}

We will use this lemma to prove the next lemma.
\begin{lem} \label{intgegen}
Let $k>\frac12$. For all $n\geq0$, we have
\[
\int_{-1}^1 G^{(k)}_n(x) (1-x^2)^{k-\frac{3}{2}}dx=
\frac{\Gamma(k-\frac{1}{2})\sqrt\pi}{\Gamma(k)},\ \text{ if $n$ is even},
\]
and 0 if $n$ is odd.
\end{lem}
\begin{proof}
Since $G^{(k)}_n$ is odd for $n$ odd, we just have to look at the even case.
The explicit form (\ref{explicitgegen}) of Gegenbauer polynomials is given  by
\[
G^{(k)}_{2n}(x)=\frac{1}{\Gamma(k)}\sum_{m=0}^n(-1)^m\frac{\Gamma(k+2n-m)}{m!\Gamma(2n-2m)!} (2x)^{2n-2m}.
\]
Hence, by parity,
\begin{align*}
\int_{-1}^1G^{(k)}_{2n}(x&) (1-x^2)^{k-\frac{3}{2}}dx\\
&=2\int_0^1G^{(k)}_{2n}(x) (1-x^2)^{k-\frac{3}{2}}dx\\
&=\frac{1}{\Gamma(k)}\sum_{m=0}^n(-1)^m\frac{\Gamma(k+2n-m)}{m!(2n-2m)!} 2^{2n-2m}
2\int_0^1 x^{2n-2m}(1-x^2)^{k-\frac{3}{2}}dx.
\end{align*}
Making the change of variable $u=x^2$ in the integral, we obtain
\[
2\int_0^1 x^{2n-2m}(1-x^2)^{k-\frac{3}{2}}dx=\int_0^1 u^{n-m-\frac{1}{2}}(1-u)^{k-\frac{3}{2}}dx=\frac{\Gamma(n-m+\frac{1}{2})\Gamma(k-\frac{1}{2})}{\Gamma(k+n-m)},
\]
by the definition of the Beta distribution.
So,
\begin{align*}
\int_{-1}^1G^{(k)}_{2n}&(x) (1-x^2)^{k-\frac{3}{2}}dx\\
&=\sum_{m=0}^n(-1)^m \frac{\Gamma(k+2n-m)}{m!(2n-2m)!}\frac{\Gamma(n-m+\frac{1}{2})}{\Gamma(k+n-m)} 2^{2n-2m} \frac{\Gamma(k-\frac{1}{2})}{\Gamma(k)}.
\end{align*}
Using duplication formula of the Gamma function, \ie
\[
\Gamma(z)\Gamma(z+\frac{1}{2})=2^{1-2z}\sqrt\pi\Gamma(2z),
\]
we obtain
\[
\Gamma(n-m+\frac12)=2^{-2(n-m)}\sqrt\pi\frac{(2n-2m)!}{(n-m)!},
\]
hence,
\[
\int_{-1}^1G^{(k)}_{2n}(x) (1-x^2)^{k-\frac{3}{2}}dx
=\sum_{m=0}^n(-1)^m\frac{1}{m!(n-m)!}\frac{\Gamma(k+2n-m)}{\Gamma(k+n-m)}\frac{\Gamma(k-\frac{1}{2})\sqrt\pi}{\Gamma(k)}.
\]
Using the rising factorial notation, and the change of index $j=n-m$, this can be rewritten  
\[
\int_{-1}^1G^{(k)}_{2n}(x) (1-x^2)^{k-\frac{3}{2}}dx
=\frac{1}{n!}\sum_{j=0}^n(-1)^{n-j} \binom{n}{j}(k+j)_n\frac{\Gamma(k-\frac{1}{2})\sqrt\pi}{\Gamma(k)}.
\]
So, to prove the lemma, we have to show that
\[
I_n=\sum_{j=0}^n(-1)^{n-j}\binom{n}{j}(k+j)_n=n!.
\]
This is done by induction. The cases 0 and 1 are easily checked. Suppose it is true for $n\geq1$. Introduce
\[
\tilde I_n=(-1)^nI_n=\sum_{j=0}^n(-1)^{j}\binom{n}{j}(k+j)_n.
\]
Then, since $(k+j)_{n+1}=(k+j)_n(k+j+n)$, we have
\[
\tilde I_{n+1}=(k+n)\sum_{j=0}^{n+1}(-1)^j\binom{n+1}{j}(k+j)_n
+\sum_{j=0}^{n+1}(-1)^j\binom{n+1}{j}j(k+j)_n.
\]
The first term in the right hand side of the above equality  is 
\[
(k+n)\sum_{j=0}^{n+1}(-1)^j\binom{n+1}{j}[k+j+n-1]_n=0,
\]
by lemma \ref{lemmaiwa} applied to $\alpha=k+n-1$. So, we have 
\[
\tilde I_{n+1}=
\sum_{j=0}^{n+1}(-1)^j\binom{n+1}{j}j(k+j)_n=
\sum_{l=0}^n(-1)^{l+1}(n+1)\binom{n}{l}(k+l+1)_n.
\]
Now, developing the rising factorial as
\[
(a+b)_n=\sum_{j=0}^n\binom{n}{j}(a)_j(b)_{n-j},
\]
(this can be seen by looking at the $n^\text{th}$-moment of the sum of two independent random variables with Gamma distributions of parameters $a$ and $b$), we get
\[
(k+l+1)_n=\sum_{j=0}^n\binom{n}{j}(k+l)_j(1)_{n-j}
=\sum_{j=0}^n \frac{n!}{j!}(k+l)_j,
\]
since $(1)_{n-j}=(n-j)!$. Hence, we have
\begin{align*}
\tilde I_{n+1}&=
\sum_{j=0}^n \frac{n!}{j!}(n+1)\left(\sum_{l=0}^n(-1)^{l+1}\binom{n}{l}(k+l)_j\right).
\end{align*}
But lemma \ref{lemmaiwa} gives
\[
\sum_{l=0}^n(-1)^l\binom{n}{l}(k+l)_j=\sum_{l=0}^n(-1)^l\binom{n}{l}[k+j-1+l]_j=0
\]
for $j=1,\ldots,n-1$. Hence, the only non-zero term  is for $j=n$, and 
\begin{align*}
\tilde I_{n+1}&=(n+1)(-1)\sum_{l=0}^n(-1)^l\binom{n}{l}(k+l)_n\\
&=(-1)^{n+1}(n+1)!,
\end{align*}
by induction hypothesis. So, we obtain $I_n=(n+1)!$, which proves the assertion, and the lemma follows.
\end{proof}

First, we make the following remark.
\begin{lem}
Let $(X_t)_{t\geq0}$ be the radial affine Dunkl process, and $k>\frac{1}{2}$. Then, for all $x\in]0,1[$, and all $t\geq0$, we have
\[
\E_x\left(\frac{1}{\sin^2(\pi X_t)}\right)<+\infty.
\]
\end{lem}
\begin{proof}
We have,
\begin{align*}
\E_x\left(\frac{1}{\sin^2(\pi X_t)}\right)&=\int_{[0,1]} \frac{1}{\sin^2(\pi y)}q_t(x,y)dy\\
&=\int_{[0,1]} \sum_{n\geq0}e^{-\lambda_n t}G_n^{(k)}(\cos\pi x)G_n^{(k)}(\cos\pi y)\omega_n^{(k)}(\sin\pi y)^{2k-2}dy,
\end{align*}
which is integrable as soon as $2k-2>-1$, \ie $k>\frac{1}{2}$, the summability  of the series being guaranteed by the term $e^{-\lambda_n t}$.
\end{proof}
Note that the proof of this lemma gives that for $k=\frac{1}{2}$, $\E_x\left(\frac{1}{\sin^2(\pi X_t)}\right)=+\infty$ for all $t>0$.
Actually, lemma \ref{intgegen} leads to the following proposition. 
\begin{prop} \label{propsin}
Let $k>\frac{1}{2}$. 
For all $x\in]0,1[$, and all $t\geq0$, we have
\[
\E_x \left(\int_0^t \frac{ds}{\sin^2(\pi X_s)}\right)<+\infty.
\]
\end{prop}
\begin{proof}
For all $t>0$, we have,
\begin{align*}
\int_0^t\E_x&\left(\frac{ds}{\sin^2(\pi X_s)}\right)\\
&=\int_0^t\int_0^1 \frac{1}{\sin^2(\pi y)}q_s(x,y)dyds\\
&=\int_0^t\int_0^1 \sum_{n\geq0}e^{-\lambda_n s}G_n^{(k)}(\cos\pi x)G_n^{(k)}(\cos\pi y)\omega_n^{(k)}(\sin\pi y)^{2k-2}dyds\\
&=\sum_{n\geq0}\frac{1}{\lambda_n}(1-e^{-\lambda_n t})G_n^{(k)}(\cos\pi x)\omega_n^{(k)}
\int_0^1G_n^{(k)}(\cos\pi y)(\sin\pi y)^{2k-2}dy.
\end{align*}
First, remark that the term $n=0$ is not a problem since $G^{(k)}_0(y)=1$ and $\omega^{(k)}_0=\frac{k\Gamma(k)^2}{2^{1-2k}\Gamma(2k)}$. 
Now, by the change of variables $u=\cos(\pi y)$, and lemma \ref{intgegen}, we have
\[
\int_0^1G_n^{(k)}(\cos\pi y)(\sin\pi y)^{2k-2}dy=\frac{1}{\pi}\int_{-1}^1G^{(k)}_n(u)(1-u^2)^{k-\frac32}du=\frac{1}{\sqrt \pi}\frac{\Gamma(k-\frac12)}{\Gamma(k)},
\]
if $n$ is even, and 0 if $n$ is odd. So,
\[
\int_0^t\E_x\left(\frac{ds}{\sin^2(\pi X_s)}\right)
=\sum_{\substack{n\geq0\\ n\text{ even}}} \frac{1}{\lambda_n}(1-e^{-\lambda_n t})G^{(k)}_n(\cos\pi x)
\omega^{(k)}_n \frac{1}{\sqrt\pi}\frac{\Gamma(k-\frac12)}{\Gamma(k)}.
\]
Hence, it suffices to prove that 
\[
\sum_{\substack{n\geq0\\ n\text{ even}}} \left|\frac{1}{\lambda_n}G^{(k)}_n(\cos\pi x)
\omega^{(k)}_n\right| <+\infty.
\]
Using Stirling's formula for the gamma function, \ie
\[
\Gamma(z)=\frac{\sqrt{2\pi}}{\sqrt z}z^ze^{-z}\left(1+\operatorname{O}\!\Big(\frac{1}{z}\Big)\right), 
\]
we find that 
\[
\omega^{(k)}_n\underset{+\infty}{\sim} n^{2-2k}.
\]
Now, using asymptotic expansion of Gegenbauer polynomials (see \cite{magnus}), that is
\[
G^{(k)}_n(\cos \pi x)=2^{1-k}\frac{\Gamma(n+k)}{n!\Gamma(k)}(\sin\pi x)^{-k}\cos\left((n+k)\pi x-k\pi^2/2\right) 
+\operatorname{O}\!\left(n^{k-2}\right),
\]
for $0<x<1$, we have (recall that $\lambda_n=\frac{\pi^2}{2}n(n+2k)$),
\[
\left|\frac{1}{\lambda_n}\omega^{(k)}_n G^{(k)}_n(\cos\pi x)\right|\underset{+\infty}{\sim}\frac{1}{n^{k+1}},
\]
Hence, since $k>\frac12$, the series is convergent, which proves the proposition.
\end{proof}

\begin{rmq} 
We obviously obtain the same results if a.s. $X_0=x$ for some $x\in\R\setminus\Z$ (not only in $]0,1[$), that is for all $x\in\R\setminus\Z$, $t\geq0$, and $k>\frac{1}{2}$,
\[
\E_x\left(\int_0^t\frac{ds}{\sin^2(\pi X_s)}\right)<+\infty.
\]
\end{rmq}


Due to the importance of the process $\frac{1}{\sin^2(\pi X_\cdot)}$ for the construction of the affine Dunkl process as we will see in the next section, we put the following definition.
\begin{defi} \label{defeta}
For all $t\geq0$, we define
\[
\eta_t=\frac{k}{2}\pi^2\int_0^t \frac{ds}{\sin^2(\pi X_s)}.
\]
\end{defi}
Let us summarize the properties of the process $\eta$. For all $x\in\R\setminus\Z$, it is a $\P_x$-almost surely  finite continuous increasing process, with $\eta_0=0$ and $\eta_t\to+\infty$ a.s. when $t$ goes to infinity. Furthermore, it has finite expectation for 
$k>\frac{1}{2}$, and for $k=\frac{1}{2}$, $\E_x(\eta_t)=+\infty$ for $t>0$.

Now we can pass to the construction properly  speaking of the affine Dunkl process.


\section{The affine Dunkl process}\label{sectionaffinedunkl}

We will define in this section the affine Dunkl process with parameter $k$ ($k\geq\frac12$) as the Markov process with infinitesimal generator
\[
\A f(x)=\frac{1}{2}f''(x) +k\pi\cot(\pi x) f'(x) + \frac{k}{2}\sum_{n\in \Z}\frac{f(s_n(x))-f(x)}{(x-n)^2},
\]
for $f\in C^2_b(\R)$ and $x\in\R\setminus\Z$. To achieve this, we will start with the radial affine Dunkl process living in some  alcove, and add the jumps at some random times. This is a kind of skew-product decomposition as the one done in \cite{chyb} (see also \cite{schapiraskew} for the same decomposition in the Heckman-Opdam setting). More precisely, we construct a pure jump process on the affine  Weyl group $W$, and use the action of $W$ on the radial process.

\subsection{Jump process on the affine Weyl group}

First, using what we have seen in (\ref{defeta}), that is $\eta_t=\frac{k}{2}\pi^2\int_0^t\frac{ds}{\sin^2(\pi X_s)}$ is an a.s. finite continuous increasing process, with $\eta_0=0$ and $\eta_t\to+\infty$ as $t\to+\infty$, we have that $\eta_t$ is a well-defined time change. Let us call its inverse $a(t)$, \ie
\[
a(t)=\inf\{s\geq 0 \, |\, \eta_s>t\},
\]
so $a$ is continuous, increasing, $a(0)=0$, $a(t)<+\infty$ for all $t\geq0$, and $a(t)\to+\infty$ as $t\to+\infty$, a.s. It is well-known that such time-change transformation preserves the Markovian character of a process (see for example \cite{dynkin}), so the process
\[
\tilde X_t=X_{a(t)}
\]
is a strong Markov process, with infinitesimal generator
\[
\tilde \L f(y)=\frac{2\sin^2(\pi y)}{k\pi^2}\L f(y),
\]
for $f\in C(\bar I)\cap C^2(I)$ and $y\in I$, where $I$ is the alcove containing $X_0$.

Recall that  for $x\in\R\setminus\Z$, we have the series expansion
\[
\frac{\pi^2}{\sin^2(\pi x)}=\sum_{n\in\Z}\frac{1}{(x-n)^2}.
\] 
 We denote by $\sigma^x$ the following probability measure on the affine Weyl group $W$
\[
\sigma^x(dw)=\sum_{n\in\Z}\frac{\sin^2(\pi x)}{\pi^2}\frac{1}{(x-n)^2}\delta_{s_n}(dw).
\]
Let $(N_t)_{t\geq0}$ be a Poisson point process  with intensity 1, independent of $X$, \ie
\[
N_t=\sum_{n\geq1}\1_{\{\tau_n\leq t\}},
\]
where $\tau_0=0$  and $\tau_n=\sum_{j=1}^n e_j$, where $(e_j)_{j\geq1}$ is a sequence of independent and identically distributed random variables, with exponential distribution of parameter 1, and independent of $X$.

Now define recursively the processes $\tilde X^j$ and the random variables $(\beta_j)_{j\geq1}$ on $W$ by
\begin{equation}\label{Xn}
\tilde X^j_t=\beta_j\point\tilde X^{j-1}_t,
\end{equation}
 for all $j\geq1$, with $\tilde X^0_t=\tilde X_t$, and where   conditionally 
 on $\{\tilde X^{j-1}_{\tau_j}=x\}$, $\beta_j$ is distributed according to $\sigma^x$. Note  that $\tilde X^j$ is the Markov process $\tilde X$ with initial condition $\tilde X^j_0=\beta_j\cdots\beta_1\point\tilde X_0$.

Using left multiplication on $W$, we define the jump process on $W$ 
\begin{equation}\label{jumpprocessW}
w_t=\xi_{N_t}=\beta_n\cdots\beta_1, \quad \text{ for $t\in[\tau_n,\tau_{n+1}[$},
\end{equation}
where $\xi_n=\beta_n\cdots\beta_1$, for all $n\geq1$.

\subsection{Skew-product decomposition}

Given an operator $\A$ with domain $\mathcal D(\A)$, we say that a c\`adl\`ag stochastic process $(Y_t)_{t\geq0}$  is  a solution of the martingale problem for $\A$ if for all $u\in\mathcal D(\A)$,
\[
u(Y_t)-u(Y_0)-\int_0^t\A u(Y_s)ds
\]
is a $(\mathcal F^Y_t)$-martingale,  where $(\mathcal F^Y_t)_{t\geq0}$ is the natural filtration of $Y$ (see \cite{ek} for a detailed  exposition of the theory of martingale problems).

Now we can state the main result of this paper.
\begin{thm}
Let $(X_t)_{t\geq0}$ be the radial affine Dunkl process with $X_0=x\in\R\setminus\Z$ a.s., and parameter $k\geq\frac12$, and $(w_t)_{t\geq0}$ be the pure jump process defined by (\ref{jumpprocessW}). Then the process $(Y_t)_{t\geq0}$ defined by
\[
Y_t=w_{\eta_t}\!\cdot\! X_t,
\]
with $\eta_t=\frac{k}{2}\pi^2\int_0^t\frac{ds}{\sin^2(\pi X_s)}$, is a Markov process on $\R$, with infinitesimal generator $\A$ given by
\[
\A f(y)=\frac{1}{2}f''(y)+k\pi\cot(\pi y)f'(y) + \frac{k}{2}\sum_{p\in \Z}\frac{f(s_p(y))-f(y)}{(y-p)^2},
\]
for $f\in C^2_b(\R)$ and $y\in\R\setminus\Z$, and such that $Y_t\in\R\setminus\Z$ for all $t\geq0$ a.s.
\end{thm}
\begin{defi}
The process $(Y_t)_{t\geq0}$ defined in the  above theorem is called the \emph{affine Dunkl process} with parameter $k$.
\end{defi}
\begin{proof}
Let us call $I$ the alcove containing $x$, that is $I=]\lfloor x\rfloor, \lfloor x \rfloor +1[$. By proposition \ref{eqradial} and remark \ref{eqradialremark}, we have that $X$ lives in $I$ almost surely.
Consider the process $\tilde X$, with generator $\tilde \L$, defined previously by $\tilde X_t=X_{a(t)}$, where $a(t)$ is the inverse of $\eta_t$. Define, for all $t\geq0$,
\[
\tilde Y_t=w_t\point \tilde X_t.
\]
Hence, for $t\in [\tau_n,\tau_{n+1}[$, we have 
\begin{align*}
\tilde Y_t&=\beta_n\cdots\beta_1\point \tilde X_t\\
&=\tilde X^n_t,
\end{align*}
where the processes $\tilde X^n$ are defined by (\ref{Xn}).

By construction $(\tilde Y_t)_{t\geq0}$ is a c\`adl\`ag process which jumps at the random times $\tau_n$'s.
We denote by $(\mathcal {\tilde F}_t)_{t\geq0}$ the natural filtration of $\tilde Y$, and by $\mathcal {\tilde F}_t^n=\sigma(\tilde X^n_s, s\leq t)\vee \sigma(N_s, s\leq t)$.
As we shall see, $(\tilde Y_t)_{t\geq0}$ is a solution of the martingale problem for the generator $\tilde \A$ given by
\[
\tilde \A f(y)=\tilde \L f(y) + \int_W \left( f(w\!\cdot\! y)-f(y)\right)\sigma^y(dw),
\]
acting on $f\in C^2_b(\R)$ for $y\in \R\setminus\Z$,
where $\sigma^x(dw)=\sum_{n\in\Z}\frac{\sin^2(\pi x)}{\pi^2}\frac{1}{(x-n)^2}\delta_{s_n}(dw)$, and $\tilde \L$ is the generator  of $\tilde X$.
The proof follows exactly the lines of proposition $10.2$, chapter  4 of \cite{ek}, see also lemma 17 in \cite{chyb}. First, since $(\tilde X^n_t)_{t\geq \tau_n}$ is a Markov process with generator $\tilde \L$, and using independence of $\tilde X$ and $(\tau_{n})_{n\geq0}$, 
we have that  for $u\in C^2_b(\R)$ 
\[
u(\tilde X^n_{(t\vee \tau_n)\wedge \tau_{n+1}})-u(\tilde X^n_{\tau_n})-\int_{\tau_n}^{(t\vee \tau_n)\wedge \tau_{n+1}} \L u(\tilde X^n_s)ds
\]
is a $(\mathcal {\tilde F}_t)$-martingale. 
Hence, summing  over $n\geq0$, and using 
\[
u(\tilde X^n_{(t\vee \tau_n)\wedge \tau_{n+1}})
=u(\tilde X^n_{\tau_n})\1_{\{t<\tau_n\}}
+u(\tilde X^n_t)\1_{\{\tau_n\leq t <\tau_{n+1}\}}
+u(\tilde X^n_{\tau_{n+1}})\1_{\{t\geq \tau_{n+1}\}},
\]
we get that
\begin{equation}\label{martingale1}
u(\tilde Y_t)-u(\tilde Y_0)-\int_0^t \L u(\tilde Y_s)ds -\sum_{n=1}^{N_t} \left( u(\tilde X^n_{\tau_n})-u(\tilde X^{n-1}_{\tau_n})\right)
\end{equation}
is a $(\mathcal {\tilde F}_t)$-martingale.  Note that, since $N$ is a Poisson process,  we have that
\[
\int_0^t \left(\int_W u(w\point\tilde Y_{s^-})\sigma^{\tilde Y_{s^-}}(dw)-u(\tilde Y_{s^-})\right)d(N_s-s),
\]
is a $(\mathcal {\tilde F}_t)$-martingale, and is equal to
\begin{equation}\label{martingale3}
\begin{split}
\sum_{n=1}^{N_t}&\left(\int_W u(w\point  \tilde X^{n-1}_{\tau_n})\sigma^{\tilde X^{n-1}_{\tau_n}}(dw)-u(\tilde X^{n-1}_{\tau_n})\right) \\ &\qquad\qquad\qquad
-\int_0^t \left(\int_W u(w\point\tilde Y_{s^-})\sigma^{\tilde Y_{s^-}}(dw)-u(\tilde Y_{s^-})\right)ds
.
\end{split}
\end{equation}
$$ $$
But 
\begin{equation}\label{martingale2}
\sum_{n=1}^{N_t}\left(u(\tilde X^n_{\tau_n}) -
\int_W u(w\point  \tilde X^{n-1}_{\tau_n})\sigma^{\tilde X^{n-1}_{\tau_n}}(dw)\right)
\end{equation}
is also a $(\mathcal {\tilde F}_t)$-martingale. This follows from the fact that for all $t_1<\cdots<t_m\leq s<t$, and all $h_1,\ldots,h_m$ measurable bounded functions, we have
\begin{align*}
&\E\left(\prod_{i=1}^mh_i(Y_{t_i}) \sum_{n\geq1}\1_{\{s<\tau_n\leq t\}}\left(u(\tilde X^n_{\tau_n}) -
\int_W u(w\point  \tilde X^{n-1}_{\tau_n})\sigma^{\tilde X^{n-1}_{\tau_n}}(dw)\right) \right)\\
&=\E\left(\prod_{i=1}^mh_i(Y_{t_i}) \sum_{n\geq1}\1_{\{s<\tau_n\leq t\}}
\E\left( u(\tilde X^n_{\tau_n}) -
\int_W u(w\point  \tilde X^{n-1}_{\tau_n})\sigma^{\tilde X^{n-1}_{\tau_n}}(dw) \Big| \, 
\mathcal {\tilde F}^{n-1}_{\tau_n}                           \right)                            \right) \\
&=0,
\end{align*}
since $\tilde X^n_{\tau_n}=\beta_n\point\tilde X^{n-1}_{\tau_n}$, and $\beta_n$ is distributed according to $\sigma^{\tilde X^{n-1}_{\tau_n}}$ conditionally to $\tilde X^{n-1}_{\tau_n}$.

Adding (\ref{martingale2}) and (\ref{martingale3}) to (\ref{martingale1}), we get that
\[
u(\tilde Y_t)-u(\tilde Y_0)-\int_0^t \L u(\tilde Y_s)ds -\int_0^t \int_W \left( u(w\point \tilde Y_s) -u(\tilde Y_s)\right)\sigma^{\tilde Y_s}(dw)ds
\]
is a $(\mathcal {\tilde F}_t)$-martingale (since $\tilde Y$ is c\`adl\`ag and $\{s\, |\, \tilde Y_{s^-}\not = \tilde Y_s\}$ is Lebesgue negligible, we can replace $\tilde Y_{s^-}$ by $\tilde Y_s$ in the last integral). Hence, we have obtained that for $u\in C^2_b(\R)$
\[
u(\tilde Y_t)-u(\tilde Y_0)-\int_0^t \tilde \A u(\tilde Y_s)ds
\]
is a $(\mathcal {\tilde F}_t)$-martingale, which proves that $(\tilde Y_t)_{t\geq 0}$ is a solution of the martingale problem for the generator $\tilde \A$.

We are now interested in solution of
\begin{equation}\label{solchange}
Y_t= \tilde Y\left(\int_0^t\beta(Y_s)ds\right),
\end{equation}
with $\beta(y)=\frac{k\pi^2}{2\sin^2(\pi y)}$ for $y\in\R\setminus\Z$, and $\beta(y)=0$ for $y\in\Z$. 
First remark that since $\sin^2(\pi \cdot)$ is a $W$-invariant function, we have
$\beta (\tilde Y_t)=\beta (\tilde X_t)$ almost surely, since $\tilde Y_t=w_t\point \tilde X_t$. Since $X$ never touches 0 and 1 a.s., so is $\tilde X$, and $\beta \circ \tilde X$ is a.s. bounded on bounded intervals. Define
\begin{align*}
\zeta_1&=\inf\left\{t\geq 0\ | \int_0^t \frac{ds}{\beta(\tilde Y_s)}=+\infty \right\}\\
\intertext{and,}
\zeta_0&=\inf\left\{t\geq0\ |\ \beta(\tilde Y_t)=0 \right\}.
\end{align*}
Since $0<\sin^2(\pi \tilde X_t)<1$ for all $t\geq0$, we easily see that $\zeta_1=\zeta_0=+\infty$, so by applying theorem $1.3$, chapter 6 of \cite{ek}, we have that equation (\ref{solchange}) admits a solution $(Y_t)_{t\geq0}$, which is a solution of the martingale problem  for $\A=\beta \tilde \A$, \ie for all $u\in C^2_b(\R)$,
\[
u(Y_t)-u(Y_0)-\int_0^t \A u(Y_s)ds
\]
is a $(\mathcal {\tilde F}_{\tau(t)})$-martingale, where $\tau(t)=\int_0^t\beta(Y_s)ds$. Since we have
\begin{align*}
\tilde \A u(y)&=\tilde \L u(y) + \int_W \left( u(w\!\cdot\! y)-u(y)\right)\sigma^y(dw)\\
&= \frac{2\sin^2(\pi y)}{k\pi^2}\L u(y) +\frac{\sin^2(\pi y)}{\pi^2}\sum_{n\in\Z}\frac{u(s_n(y))-u(y)}{(y-n)^2},
\end{align*}
we obtain
\[
\A u(y)=\L u(y) + \frac{k}{2}\sum_{n\in \Z}\frac{u(s_n(y))-u(y)}{(y-n)^2},
\]
acting on $u\in C^2_b(\R)$, for $y\in\R\setminus\Z$. Since $\beta(Y_s)=\beta(X_s)$ for all $s\geq0$ by the $W$-invariance of $\sin^2(\pi\cdot)$, we have that
\[
\int_0^t\beta(Y_s)ds=\eta_t,
\]
for all $t\geq0$. Hence, we obtain the  skew-product decomposition of $(Y_t)_{t\geq0}$ given by
\[
Y_t=w_{\eta_t}\!\cdot\! X_t,
\]
for all $t\geq0$.
Let $\pi$ be the projection onto the principal alcove $\mathcal A_0=]0,1[$. Then, by the invariance of $\pi$ under the action of the Weyl group $W$ and the skew-product representation of the affine Dunkl process, we see that
\[
\pi(Y_t)=X_t\  \text{ a.s.},
\]
\ie $(\pi(Y_t))_{t\geq0}$ is the radial affine Dunkl process. Hence, if two process $Y$ and $Y'$ are solutions to the martingale problem for $\A$, then $\pi(Y_t)=\pi(Y'_t)=X_t$, so $Y$ and $Y'$ have the same one-dimensional distributions, and by the same arguments as in theorem $4.2$, chapter 4 in \cite{ek}, we have that the affine Dunkl process $Y$ is a Markov process with infinitesimal generator $\A$. By construction, $Y$ is c\`adl\`ag and  lives a.s. in $\R\setminus\Z$, else the radial process $X$ would touch the walls of the principal alcove, which is impossible by proposition (\ref{eqradial}).
\end{proof}


\subsection{Jumps of the affine Dunkl process}

By the construction of the affine Dunkl process, we have the skew-product decomposition
\[
Y_t=w_{\eta_t}\point X_t,
\]
for $t\geq0$.
This shows that there is a jump of the process at time $t$ when the functional $\eta_t$ is equal to one of the $\tau_n$'s. Hence, 
the number of jumps $V_t$ of $(Y_t)_{t\geq0}$ before time $t$, \ie
\[
V_t=\sum_{s\leq t}\1_{\{\Delta Y_s\not=0\}},
\] 
where $\Delta Y_s=Y_s-Y_{s^-}$, is exactly given by the point process 
\[
V_t=\sum_{n\geq1}\1_{\{\eta_t\geq \tau_n\}}.
\]
Since $\eta$ is a well-defined time change, 
we get the following representation of $V$ in term of a time-change Poisson process
\[
V_t=N_{\eta_t},
\]
for all $t\geq0$, where $(N_t)_{t\geq0}$ is the Poisson process considered previously.
Using this representation and the fact that for all $t\geq0$, $\eta_t<+\infty$ a.s., we get immediately  the following proposition.
\begin{prop}
For all $x\in\R\setminus\Z$, and all  $t\geq0$
\[
V_t<+\infty, \quad \text{ $\P_x$-almost surely},
\]
\ie the number of jumps of the affine Dunkl process in a finite time interval  is almost surely finite.
\end{prop}
 
Now define for all $n\geq1$,
\[
T_n=a(\tau_n),
\]
where $a$ is the inverse of $\eta$, \ie $a(t)=\inf\{s\geq0\ |\ \eta_s>t\}$, and $T_0=0$. The sequence $(T_n)_{n\geq1}$ corresponds to  the jump times of the affine Dunkl process. Since $a$ is increasing and $a(t)\to+\infty$ as $t\to+\infty$ a.s., we have that for all $n\geq0$,
\[
T_n>T_{n-1}, \quad \text{ $\P_x$-almost surely},
\]
and $T_n\to+\infty$ when $n\to+\infty$ a.s. Note that we can also define the jump times recursively by
\begin{equation}\label{autretempssaut}
T_n=\inf\left\{t\geq T_{n-1} \big|\ \frac{k\pi^2}{2}\int_{T_{n-1}}^t \frac{ds}{\sin^2(\pi X_s)}>e_n\right\},
\end{equation}
where $(e_n)_{n\geq1}$ is a sequence of independent and identically distributed random variables with exponential distribution of parameter 1. 
Indeed,  define for all $n\geq1$ and all $t\geq0$,
\[
\eta_n(t)=\frac{k\pi^2}{2}\int_{T_{n-1}}^{T_{n-1}+t}\frac{ds}{\sin^2(\pi X_s)},
\]
so $\eta_n$ is a well-defined  time-change with inverse given by
\[
\eta_n^{-1}(t)=\inf\{s\geq 0\, |\, \eta_n(s)>t\},
\]
which is finite for all $t\geq0$, and goes to infinity as $t$ goes to infinity. Then, we can rewrite $T_n$ as
\begin{align*}
T_n&=T_{n-1}+\inf\left\{t\geq0\ |\ \eta_n(t)>e_n\right\}\\
&=T_{n-1}+\eta_n^{-1}(e_n).
\end{align*}
 So,
\[
T_n-T_{n-1}=\eta_n^{-1}(e_n),
\]
which gives $\eta_n(T_n-T_{n-1})=e_n$, and
\[
\eta_{T_n}=\sum_{j=1}^n \eta_j(T_j-T_{j-1})=\sum_{j=1}^n e_j=\tau_n,
\]
or equivalently $T_n=a(\tau_n)$. Note that the fact that the sequence $(T_n)_{n\geq0}$ is well-defined can be proved directly using expression (\ref{autretempssaut}) and the strong Markov property of $(X_t)_{t\geq0}$.

Since $V_t$ is a time-change Poisson process, it is not difficult to exhibit its compensator.
\begin{lem}
Let $V_t=N_{\eta_t}$. The compensator of $V$ is $\eta$, that is
\[
V_t-\eta_t
\]
is a martingale with respect to the filtration $(\mathcal F_{\eta_t}^N)_{t\geq0}$, where $\mathcal F^N_t=\sigma(N_s,s\leq t)$.
\end{lem}
\begin{proof}
Let $0\leq s < t$. Define $T=\eta_t$ and $S=\eta_s$. We have $S<T$, since $t\mapsto \eta_t$ is increasing. Note also that $S$ and $T$ are stopping times with respect to $\mathcal F_t$. Hence, since $N$ is a Poisson process, $N_t-t$ is a martingale and by the optional sampling theorem, we have
\[
\E_x\left(V_t-\eta_t \,|\, \mathcal F_{\eta_s}^N\right)
=\E_x\left(N_T-T \, |\,  \mathcal F_S^N\right)
=N_S-S
=V_s-\eta_s. \qedhere
\]
\end{proof}

Since $\eta$ is the compensator of $V$, we have by proposition \ref{propsin} that 
\[
\E_x(V_t)<+\infty, \quad \text{for $k>\frac{1}{2}$,}
\] 
 and $\E_x(V_t)=+\infty$ for $k=\frac{1}{2}$, that is, the number of jumps of the affine Dunkl process has finite expectation when $k>\frac12$, and infinite expectation for $k=\frac12$.

\subsection{Martingale decomposition} 

First, we remark that $Y$ is a local martingale.
\begin{prop}
The affine Dunkl process $(Y_t)_{t\geq0}$ is a local martingale.
\end{prop}
\begin{proof}
Using the formula
\[
\pi \cot(\pi x)=\sum_{n\in\Z}\frac{1}{x-n},
\]
one can see that the function $f(x)=x$ is killed by the generator $\A$ of $(Y_t)_{t\geq0}$, which proves that $Y$ is a local martingale.
\end{proof}

We give now the martingale decomposition of $Y$ into its continuous and purely discontinuous parts. First, recall that the 
L\'evy kernel $N$ of a Markov process describes the distribution of its jumps, see \cite{meyerlevykernel}. For all $x\in\R$, and for a function $f$ in the domain of the infinitesimal generator which vanishes in a neigbourhood of $x$, the L\'evy kernel $N$ of $(Y_t)_{t\geq0}$ is given by
\[
\A f(x)=\int_\R N(x,dy)f(y) .
\]
Hence, by the explicit form of the infinitesimal generator $\A$, we get immediately that
\[
N(x,dy)=\frac{k\pi^2}{2}\sum_{n\in\Z} \frac{\delta_{s_n(x)}(dy)}{(x-n)^2},
\]
for all $x\in\R\setminus\Z$. By \cite{meyerlevykernel}, for all nonnegative measurable function $f$ on $\R^2$, the nonnegative discontinuous functional 
\[
\sum_{s\geq t}f(Y_{s^-},Y_s)\1_{\{\Delta Y_s\not=0\}},
\]
where $\Delta Y_s=Y_s-Y_{s^-}$,
can be compensated by the process
\[
\int_0^t ds \int_\R N(Y_{s^-},dy)f(Y_{s^-},y).
\]

\begin{prop}
We have the following martingale decomposition,
\[
Y_t=Y_0+B_t+M_t,
\]
where $(B_t)_{t\geq0}$ is a standard Brownian motion, and $(M_t)_{t\geq0}$ is a purely discontinuous local martingale which can written as the compensated sum of its jumps:
\[
M_t=-\sum_{s\leq t } \Delta Y_s\1_{\{\Delta Y_s\not=0\}}+\int_0^t k\pi\cot(\pi Y_s)ds.
\]
\end{prop}
The proof use It\^o's formula and the theory of L\'evy kernel and is exactly the same as in the classical case of Dunkl processes, so we refer to \cite{gallardoyor2} for more details.


\newcommand{\etalchar}[1]{$^{#1}$}
\providecommand{\bysame}{\leavevmode\hbox to3em{\hrulefill}\thinspace}
\providecommand{\MR}{\relax\ifhmode\unskip\space\fi MR }
\providecommand{\MRhref}[2]{%
  \href{http://www.ams.org/mathscinet-getitem?mr=#1}{#2}
}
\providecommand{\href}[2]{#2}

\end{document}